\documentclass[11pt]{amsart}
\usepackage[margin=1.15in]{geometry}
\usepackage{amssymb,amsmath,mathtools}
\usepackage{booktabs}
\usepackage{enumitem}
\usepackage[expansion=false]{microtype}
\usepackage[hidelinks]{hyperref}

\newtheorem{theorem}{Theorem}[section]
\newtheorem{proposition}[theorem]{Proposition}
\newtheorem{lemma}[theorem]{Lemma}
\newtheorem{corollary}[theorem]{Corollary}
\newtheorem{introthm}{Theorem}

\theoremstyle{definition}
\newtheorem{example}[theorem]{Example}
\newtheorem{problem}[theorem]{Problem}
\theoremstyle{remark}
\newtheorem{remark}[theorem]{Remark}

\numberwithin{equation}{section}

\newcommand{\N}{\mathbb{N}}
\newcommand{\Z}{\mathbb{Z}}
\newcommand{\Q}{\mathbb{Q}}
\newcommand{\F}{\mathbb{F}}
\newcommand{\cS}{\mathcal{S}}
\newcommand{\legendre}[2]{\left(\frac{#1}{#2}\right)}

\title[Squares in reciprocal-sum sequences]{Perfect squares in reciprocal-sum sequences\\ and primes that are inert in quadratic fields}
\author{Mateo Matijasevick}
\author{Santiago Rodr\'iguez}
\author{Gregorio Salazar}

\begin{document}

\begin{abstract}
Let $k$ be a positive integer, and consider the sequences of positive rationals with $x_0\in\N$ and $x_{n+1}=k/(x_0+x_1+\dots+x_n)$. Write $x_n=a_n/b_n$ in lowest terms. We show that there is a rational constant $c>0$ such that $c\,a_n+b_n$ is a perfect power for every such sequence and every $n\ge2$ if and only if $k=h^2$ and every prime factor of $h$ is congruent to $3$ modulo $4$; in that case $c=2/h$ and the powers are squares. The proof rests on the observation that $t_n=(x_0+\dots+x_n)/h$ satisfies $t_{n+1}=t_n+1/t_n$, a recursion under which reduced fractions never cancel. This yields an exact formula for $c\,a_n+b_n$, shows that along any single sequence each prime factor of $h$ spoils at most one term, and leads to two generalizations. For arbitrary $k$ the invariant $b_n^2-\frac4k a_n^2$ is always a rational square, and it is always an integer square exactly when the primes dividing the square part of $k$ satisfy an inertness condition in $\Q(\sqrt{-k})$. For the recursions $x_{n+1}=h^2/(x_0+\dots+x_n+n\mu h)$ the role of the Gaussian integers is played by the quadratic fields $\Q(\sqrt{\mu^2-4})$, which include $\Q(\sqrt{-3})$ and every real quadratic field.
\end{abstract}

\maketitle

%======================================================================
\section{Introduction}\label{sec:intro}

The following problem was proposed and solved by S. Rodriguez and G. Salazar, after they had worked on the problem of retriangular numbers \cite{Mat-retri}.

\begin{problem}\label{prob:main}
Let $k$ be a positive integer, and let $\cS_k$ be the family of sequences $x_0,x_1,x_2,\dots$ of positive rational numbers with $x_0\in\N$ and
\begin{equation}\label{eq:rec}
x_{n+1}=\frac{k}{x_0+x_1+\dots+x_n}\qquad(n\ge0).
\end{equation}
Write $x_n=a_n/b_n$ with $a_n,b_n$ coprime positive integers. For which $k$ is there a rational number $c>0$ and an integer $m\ge2$ such that $c\,a_n+b_n$ is a perfect $m$-th power for every sequence of $\cS_k$ and every $n\ge2$?
\end{problem}

For example, for $k=9$ and $x_0=1$ the sequence begins $1,9,\frac9{10},\frac{90}{109},\frac{9810}{12781},\dots$, and with $c=\frac23$ the numbers $c\,a_n+b_n$ for $n=2,3,4$ are $16=4^2$, $169=13^2$ and $19321=139^2$ (Example~\ref{ex:k9}). The answer is as follows.

\begin{introthm}\label{thmA}
There are $c\in\Q_{>0}$ and $m\ge2$ as in Problem~\ref{prob:main} if and only if $k=h^2$ for an integer $h$ all of whose prime factors are congruent to $3$ modulo $4$ (including $h=1$). In that case $m=2$ and $c=2/h$ necessarily.
\end{introthm}

The condition on $h$ is Fermat's criterion in disguise: a prime $p$ divides a sum $u^2+v^2$ of two coprime squares only if $p=2$ or $p\equiv1\pmod4$. The link is the following change of variables. Put $S_n=x_0+\dots+x_n$ and, when $k=h^2$, $t_n=S_n/h$. Since $S_{n+1}=S_n+k/S_n$,
\[
t_{n+1}=t_n+\frac1{t_n},\qquad\text{so that}\qquad \frac{u_{n+1}}{v_{n+1}}=\frac{u_n^2+v_n^2}{u_nv_n}\quad\text{if } t_n=\frac{u_n}{v_n},
\]
and the right-hand side is again in lowest terms. All the arithmetic of the problem is then concentrated in one number, $\gcd(h,u_n)$, and we obtain the exact formula (Proposition~\ref{prop:key})
\begin{equation}\label{eq:keyintro}
\frac2h\,a_{n+1}+b_{n+1}=\frac{(u_{n-1}+v_{n-1})^2}{\gcd(h,\ u_{n-1}^2+v_{n-1}^2)}\qquad(n\ge1).
\end{equation}
In the language of arithmetic dynamics, the substitution $S=ht$ conjugates $z\mapsto z+h^2/z$ to $t\mapsto t+1/t$, a rational map with good reduction at every prime \cite{Si07}. Formula~\eqref{eq:keyintro} also describes what happens along an individual sequence when $h$ is arbitrary.

\begin{introthm}\label{thmB}
Let $h\in\N$ and $k=h^2$. Along any sequence of $\cS_k$, the number $\frac2h a_n+b_n$ fails to be a perfect square for at most $\omega_1(h)$ indices $n\ge2$, where $\omega_1(h)$ is the number of distinct prime factors of $h$ that are equal to $2$ or congruent to $1$ modulo $4$. The bound is attained.
\end{introthm}

For non-square $k$ no linear expression in $a_n,b_n$ works, but a quadratic one does. Since $b-\frac{2a}{\sqrt k}$ and $b+\frac{2a}{\sqrt k}$ are conjugate in $\Q(\sqrt k)$, the natural invariant is their product.

\begin{introthm}\label{thmC}
Let $k\in\N$. For every sequence of $\cS_k$ and every $n\ge2$, the number $b_n^2-\frac4k a_n^2$ is the square of a rational number. If some $c\in\Q_{>0}$ makes $b_n^2-c\,a_n^2$ a perfect square for all sequences and all $n\ge2$, then $c=4/k$; and $b_n^2-\frac4ka_n^2$ is a perfect square for all sequences and all $n\ge2$ if and only if
\begin{enumerate}[label=\textup{(\roman*)},leftmargin=*]
\item for every odd prime $p$ with $v_p(k)=2\ell\ge2$ even, $-k/p^{2\ell}$ is a quadratic non-residue modulo $p$, and
\item if $v_2(k)=2\ell\ge4$ is even, then $k/2^{2\ell}\equiv1\pmod4$.
\end{enumerate}
\end{introthm}

Writing $k=mh^2$ with $m$ squarefree, condition (i) says that every odd prime dividing $h$ but not $m$ is inert in $\Q(\sqrt{-m})$, so Theorem~\ref{thmC} reaches every imaginary quadratic field. The smallest $k$ that fail are $18,25,45,48,72,99,100,112$.

Finally, perturbing the recursion changes the quadratic field. For $\mu\in\Z$ with $\mu\ge-1$ consider
\begin{equation}\label{eq:murec}
x_{n+1}=\frac{h^2}{x_0+x_1+\dots+x_n+n\mu h}\qquad(n\ge0),
\end{equation}
so that $\mu=0$ is Problem~\ref{prob:main} with $k=h^2$.

\begin{introthm}\label{thmD}
Let $\mu\ge-1$ with $\mu\ne2$, and suppose that $\gcd(h,\mu-2)=1$. For the sequences \eqref{eq:murec}, the number $b_n+\frac{2-\mu}{h}a_n$ is a perfect square for every $x_0\in\N$ and every $n\ge2$ if and only if $X^2+\mu X+1$ is irreducible modulo every prime factor of $h$. The general criterion, which includes the primes dividing $\mu-2$, is Theorem~\ref{thm:mu}. In particular, for $x_{n+1}=h^2/(x_0+\dots+x_n+nh)$ the number $b_n+\frac1h a_n$ is always a perfect square if and only if every prime factor of $h$ is congruent to $2$ modulo $3$.
\end{introthm}

For $p\nmid\mu^2-4$ odd, irreducibility of $X^2+\mu X+1$ modulo $p$ means that $p$ is inert in $\Q(\sqrt{\mu^2-4})$. As $\mu$ runs through the integers $\ge3$, these fields run through all real quadratic fields (Remark~\ref{rem:fields}).

\subsection*{Related work}
The asymptotic behaviour of sequences such as $a_{n+1}=a_n+1/a_n$ is a classical exercise \cite{Newman}; the orbit $1,2,\frac52,\frac{29}{10},\dots$ is recorded in \cite{OEIS} (A073833, A073834). The arithmetic of the map $z\mapsto z+k/z$, including the Galois groups of its iterates and the distribution of the primes dividing its orbits, is studied in the companion paper \cite{Mat-dyn}, which builds on the work of Jones and Manes \cite{JM14}. The retriangular problem \cite{Mat-retri}, which inspired Problem~\ref{prob:main}, leads to a different circle of ideas (affine maps between polygonal numbers), but its classification of the affine self-maps of sums of two triangular numbers rests on the same local obstruction at primes $\equiv3\pmod4$.

\subsection*{Notation} $\N=\{1,2,3,\dots\}$; $v_p$ is the $p$-adic valuation; a \emph{perfect square} is the square of an integer. For a sequence of $\cS_k$ we write $S_n=x_0+\dots+x_n$.

%======================================================================
\section{Preliminaries}\label{sec:prelim}

Since $S_{n+1}=S_n+x_{n+1}=S_n+k/S_n$, we have
\begin{equation}\label{eq:f}
S_n=f^{\,n}(x_0),\qquad f(z)=z+\frac kz.
\end{equation}
For $n\ge1$, $S_{n-1}=k/x_n$, so $x_{n+1}=\psi(x_n)$ with $\psi(x)=kx/(x^2+k)$. Because $\psi(k/x)=\psi(x)$ and $x_1=k/x_0$, also $x_2=\psi(x_0)$. Thus, if $(A,B)=(a_n,b_n)$ for some $n\ge1$, or $(A,B)=(x_0,1)$, the next term is
\begin{equation}\label{eq:step}
\frac{kAB}{A^2+kB^2}\quad\text{before reduction, and we put}\quad g=\gcd(kAB,\,A^2+kB^2).
\end{equation}

\begin{lemma}\label{lem:gprimes}
Every prime divisor of $g$ divides $k$ and $A$ and does not divide $B$.
\end{lemma}

\begin{proof}
If $p\mid g$ and $p\nmid k$, then $p\mid AB$; if $p\mid A$ then $p\mid kB^2$, so $p\mid B$, and if $p\mid B$ then $p\mid A^2$, contradicting $\gcd(A,B)=1$. Hence $p\mid k$, so $p\mid A^2$, and $p\nmid B$.
\end{proof}

Now let $k=h^2$ and $t_n=S_n/h$. By \eqref{eq:f}, $t_{n+1}=T(t_n)$ with $T(t)=t+1/t$. Write $t_n=u_n/v_n$ in lowest terms, with $u_n,v_n>0$; initially $u_0=x_0/\gcd(x_0,h)$ and $v_0=h/\gcd(x_0,h)$.

\begin{lemma}\label{lem:nocancel}
For all $n\ge0$, $u_{n+1}=u_n^2+v_n^2$ and $v_{n+1}=u_nv_n$, and these are coprime. The integers $u_n$ are pairwise coprime, and every odd prime factor of $u_n$ with $n\ge1$ is congruent to $1$ modulo $4$.
\end{lemma}

\begin{proof}
A prime dividing $u_nv_n$ divides exactly one of $u_n,v_n$, hence not $u_n^2+v_n^2$. For $m>n$ we have $v_{n+1}\equiv0$ and therefore $u_m\equiv u_{n+1}^{2^{m-n-1}}\equiv v_n^{2^{m-n}}\pmod{u_n}$, which is prime to $u_n$. Finally, if $p$ is odd and $p\mid u^2+v^2$ with $\gcd(u,v)=1$, then $p\nmid v$ and $-1\equiv(uv^{-1})^2\pmod p$, so $p\equiv1\pmod4$.
\end{proof}

\begin{proposition}\label{prop:key}
Let $k=h^2$ and $d_n=\gcd(h,u_n)$. For all $n\ge0$,
\[
a_{n+1}=\frac{hv_n}{d_n},\qquad b_{n+1}=\frac{u_n}{d_n},
\]
and for all $n\ge1$,
\begin{equation}\label{eq:key}
\frac2h\,a_{n+1}+b_{n+1}=\frac{(u_{n-1}+v_{n-1})^2}{d_n},\qquad b_{n+1}-\frac2h\,a_{n+1}=\frac{(u_{n-1}-v_{n-1})^2}{d_n}.
\end{equation}
\end{proposition}

\begin{proof}
We have $x_{n+1}=h^2/S_n=h/t_n=hv_n/u_n$, and $\gcd(hv_n,u_n)=\gcd(h,u_n)$. Moreover $u_n\pm2v_n=u_{n-1}^2\pm2u_{n-1}v_{n-1}+v_{n-1}^2$.
\end{proof}

\begin{example}\label{ex:k9}
For $h=3$ and $x_0=1$ we have $t_0=\frac13$, $t_1=\frac{10}{3}$, $t_2=\frac{109}{30}$, $t_3=\frac{12781}{3270}$, and the sequence is $1,9,\frac9{10},\frac{90}{109},\frac{9810}{12781},\dots$. Here $\frac23\cdot9+10=16=(1+3)^2$, $\frac23\cdot90+109=169=(10+3)^2$ and $\frac23\cdot 9810+12781=19321=139^2=(109+30)^2$.
\end{example}

%======================================================================
\section{Proof of Theorem~\ref{thmA}}\label{sec:main}

\emph{Step 1: $m=2$, $k$ is a square, and $c=2/\sqrt k$.} Take $x_0=kr+1$ with $r\in\Z_{\ge0}$. Since $\gcd(x_0,k)=1$, Lemma~\ref{lem:gprimes} shows that $x_2=kx_0/(x_0^2+k)$ is reduced, so
\[
F(r):=c\,a_2+b_2=c\,kx_0+x_0^2+k=\Bigl(x_0+\frac{ck}{2}\Bigr)^2+E,\qquad E=k-\frac{c^2k^2}{4},
\]
is a perfect $m$-th power for every $r\ge0$. The function $F$ is strictly increasing and $F(r)\le\Lambda(r+1)^2$ for a constant $\Lambda$. Hence $F(0),\dots,F(N)$ are $N+1$ distinct $m$-th powers in $[1,\Lambda(N+1)^2]$, so $N+1\le(\Lambda(N+1)^2)^{1/m}$, which fails for large $N$ if $m\ge3$. Thus $m=2$; write $F(r)=s_r^2$ with $s_r\in\N$ and $L_r=x_0+ck/2$. Fix $D\in\N$ with $DL_r\in\Z$ for all $r$. Then $(s_r-L_r)(s_r+L_r)=E$ and $s_r-L_r\in\frac1D\Z$. If $E\ne0$ then $\lvert s_r-L_r\rvert\ge1/D$, so $\lvert E\rvert\ge(s_r+L_r)/D\to\infty$, which is absurd. Therefore $E=0$, i.e.\ $c=2/\sqrt k$, and since $c\in\Q$, $k=h^2$ and $c=2/h$.

\emph{Step 2: sufficiency.} Suppose that every prime factor of $h$ is $\equiv3\pmod4$; in particular $h$ is odd. For $n\ge1$, Lemma~\ref{lem:nocancel} shows that $u_n$ has no prime factor $\equiv3\pmod 4$, so $d_n=1$, and \eqref{eq:key} gives $\frac2ha_{n+1}+b_{n+1}=(u_{n-1}+v_{n-1})^2$.

\emph{Step 3: necessity.} Let $p\mid h$ with $p=2$ or $p\equiv1\pmod4$. Choose $e\in\N$ with $v_p(e^2+1)=1$: take $e=1$ if $p=2$; otherwise take $e^2\equiv-1\pmod p$, replacing $e$ by $e+p$ if $p^2\mid e^2+1$. For $x_0=he$ we have $(u_0,v_0)=(e,1)$ and $u_1=e^2+1$, so by \eqref{eq:key}
\[
\frac2ha_2+b_2=\frac{(e+1)^2}{\gcd(h,e^2+1)},
\]
whose $p$-adic valuation $2v_p(e+1)-1$ is odd. So it is not a perfect square. \qed

\begin{example}
For $h=5$ and $x_0=10$ one gets $x_2=2$ and $\frac25\cdot2+1=\frac95$. For $h=26$ and $x_0=130$ one gets $x_2=5$ and $\frac2{26}\cdot5+1=\frac{36}{26}$; here both $2$ and $13$ cancel.
\end{example}

\begin{remark}[a proof in the original coordinates]\label{rem:gcd}
The students' original proof worked with the pairs $(a_n,b_n)$ and distinguished three cases according to how $v_p(a_n)$ compares with $v_p(h)$. It can be condensed into one statement: if every prime factor of $h$ is $\equiv3\pmod4$ and $\gcd(A,B)=1$, then
\[
\gcd(h^2AB,\ A^2+h^2B^2)=\gcd(h,A)^2,\qquad\text{hence}\qquad \frac2h\,a_{n+1}+b_{n+1}=\Bigl(\frac{a_n+hb_n}{\gcd(h,a_n)}\Bigr)^2.
\]
Indeed, write $\gamma=\gcd(h,A)$, $h=\gamma h_1$, $A=\gamma A_1$. The gcd equals $\gamma^2\gcd(\gamma h_1^2A_1B,\ A_1^2+h_1^2B^2)$, and a prime $q$ dividing the last gcd cannot divide $A_1$, $h_1$ or $B$; so $q\mid\gamma\mid h$ and $-1\equiv(A_1/h_1B)^2\pmod q$, which is impossible. The three cases correspond to the $p$-adic dynamics of $\psi$: for $\kappa=v_p(k)$ and $v=v_p(x)$ with $2v\ne\kappa$ one has $v_p(\psi(x))=\max(v,\kappa-v)$, while on the ``critical circle'' $2v=\kappa$, writing $x=p^{\kappa/2}\xi$, the map reduces modulo $p$ to $\bar\xi\mapsto\bar k'\bar\xi/(\bar\xi^{\,2}+\bar k')$ with $k'=k/p^\kappa$, whose poles are $\F_p$-rational exactly when $\legendre{-k'}{p}=1$.
\end{remark}

%======================================================================
\section{Individual sequences: proof of Theorem~\ref{thmB}}\label{sec:orbit}

\begin{proposition}\label{prop:orbit}
Let $h\in\N$, $k=h^2$ and $n\ge1$.
\begin{enumerate}[label=\textup{(\alph*)}]
\item $\frac2ha_{n+1}+b_{n+1}$ is a perfect square if and only if $\gcd(h,u_n)=1$.
\item Each prime $p$ divides $u_n$ for at most one $n\ge1$.
\end{enumerate}
\end{proposition}

\begin{proof}
(a) Let $p\mid d_n$ and $(u,v)=(u_{n-1},v_{n-1})$. If $p$ is odd and $p\mid u+v$, then $p\mid(u+v)^2-(u^2+v^2)=2uv$, which is impossible; so $v_p\bigl((u+v)^2/d_n\bigr)=-v_p(d_n)<0$. If $p=2$, then $u,v$ are odd, $v_2(u^2+v^2)=1$ and $v_2(d_n)=1$, so the valuation is odd. Conversely, $d_n=1$ gives a square by \eqref{eq:key}.

(b) If $p\mid u_n$ then $p\nmid v_n$, so $u_{n+1}\equiv v_n^2\not\equiv0$ and $v_{n+1}\equiv0\pmod p$; inductively $p\mid v_m$ and $p\nmid u_m$ for all $m>n$.
\end{proof}

\begin{proof}[Proof of Theorem~\ref{thmB}]
By Proposition~\ref{prop:orbit}, each failure at an index $n+1\ge2$ is caused by a prime $p\mid\gcd(h,u_n)$; by Lemma~\ref{lem:nocancel} such a prime is $2$ or $\equiv1\pmod4$, and by Proposition~\ref{prop:orbit}(b) it causes at most one failure. For $h=x_0=10$ the sequence is $10,10,5,4,\frac{100}{29},\frac{2900}{941},\dots$ and $\frac15a_n+b_n$ takes the values $2,\frac95,49,39^2,\dots$ for $n=2,3,4,5$: there are failures at $n=2$ (caused by $2$) and $n=3$ (caused by $5$), and none afterwards.
\end{proof}

\begin{remark}\label{rem:finite}
Proposition~\ref{prop:orbit} turns the question ``which terms of a given sequence fail?'' into a question about the dynamics of $T(t)=t+1/t$ over finite fields: the prime $p\mid h$ causes a failure if and only if the orbit of $x_0/h$ modulo $p$ passes through $0$ at some step $n\ge1$. For $p\equiv3\pmod4$ this never happens, because $0$ has no preimage under $T$ in $\F_p$. The same pairwise coprimality shows that the denominators $b_n$ of any sequence of $\cS_{h^2}$ involve infinitely many primes $\equiv1\pmod4$, in the manner of Euclid's proof; that these primes nevertheless have density zero is proved in \cite{Mat-dyn}.
\end{remark}

%======================================================================
\section{Arbitrary $k$: proof of Theorem~\ref{thmC}}\label{sec:generalk}

Let $(A,B)$ and $g$ be as in \eqref{eq:step}, so that the next term is $a/b$ with $a=kAB/g$ and $b=(A^2+kB^2)/g$.

\emph{Rational squares.} The identity $(A^2+kB^2)^2-4k(AB)^2=(A^2-kB^2)^2$ gives
\begin{equation}\label{eq:quadinv}
b^2-\frac4k\,a^2=\Bigl(\frac{A^2-kB^2}{g}\Bigr)^2.
\end{equation}

\emph{Uniqueness of $c$.} With $x_0=X=kr+1$ as in Step~1 of Section~\ref{sec:main}, $b_2^2-c\,a_2^2=(X^2+k)^2-ck^2X^2=(X^2+\lambda)^2-(\lambda^2-k^2)$ with $\lambda=k-\frac{ck^2}{2}$. This must be a perfect square for all $r$, and the difference-of-squares argument forces $\lambda^2=k^2$, so $c\in\{0,4/k\}$.

\emph{The criterion.} By \eqref{eq:quadinv}, $b^2-\frac4ka^2$ is a perfect square if and only if $g\mid A^2-kB^2$. As $g\mid A^2+kB^2$, this is equivalent to $g\mid2kB^2$, and by Lemma~\ref{lem:gprimes} to $g\mid2k$. Let $p\mid g$, $\kappa=v_p(k)$, $\beta=v_p(A)\ge1$, $A=p^\beta A'$ and $k=p^\kappa k'$. Then $v_p(g)=\min(\kappa+\beta,\ v_p(A^2+kB^2))$, and
\[
v_p(g)=\begin{cases}2\beta&\text{if }2\beta<\kappa,\\ \kappa&\text{if }2\beta>\kappa,\\ \kappa+\min(\beta,\tau)&\text{if }2\beta=\kappa,\end{cases}\qquad \tau=v_p(A'^2+k'B^2).
\]
So $v_p(g)\le v_p(2k)$ can fail only if $2\beta=\kappa$ and either $p$ is odd and $\tau\ge1$, or $p=2$, $\beta\ge2$ and $\tau\ge2$. For odd $p$, $\tau\ge1$ means $-k'\equiv(A'/B)^2\pmod p$, which (i) excludes. For $p=2$, $A'$ and $B$ are odd, so $A'^2+k'B^2\equiv1+k'\pmod 8$, and $\tau\ge2$ means $k'\equiv3\pmod 4$, which (ii) excludes. This proves sufficiency. Conversely, if (i) fails at $p$, take $w\in\N$ with $w^2\equiv-k'\pmod p$ and $x_0=p^\ell w$; if (ii) fails, take $x_0=2^\ell$. In both cases the first step, with $(A,B)=(x_0,1)$, gives $b_2^2-\frac4ka_2^2\notin\Z$. \qed

\begin{example}
For $k=18$ and $x_0=3$ we get $x_2=2$ and $1-\frac4{18}\cdot4=\frac19$. Up to $1000$, exactly $98$ values of $k$ fail; every squarefree $k$ works.
\end{example}

\begin{remark}\label{rem:square}
For $k=h^2$, Theorem~\ref{thmC} holds if and only if every \emph{odd} prime factor of $h$ is $\equiv3\pmod4$, so $h$ may be even. Theorem~\ref{thmA} asks for more: by \eqref{eq:key}, the two factors $b\pm\frac2ha$ of $b^2-\frac4{h^2}a^2$ must be squares separately, and this is what excludes $2\mid h$. Condition (i) is inertness in $\Q(\sqrt{-m})$, where $m$ is the squarefree part of $k$; note that inertness differs from non-representability by $X^2+mY^2$ when the class number of discriminant $-4m$ exceeds $1$, and condition (ii) at $2$ is not an inertness condition. As in Theorem~\ref{thmB}, each prime causes at most one failure along a given sequence: after a critical step with $\tau\ge1$ the valuation $v_p(a_n)$ drops below $\kappa/2$, then jumps above $\kappa/2$ and stays there.
\end{remark}

%======================================================================
\section{Other quadratic fields: the recursions \eqref{eq:murec}}\label{sec:mu}

Fix $h\in\N$ and $\mu\in\Z$ with $\mu\ge-1$. For a sequence \eqref{eq:murec} put $S_n=x_0+\dots+x_n+n\mu h$ and $t_n=S_n/h$. Then $S_{n+1}=S_n+\mu h+h^2/S_n$, that is,
\[
t_{n+1}=T_\mu(t_n),\qquad T_\mu(t)=t+\mu+\frac1t,
\]
which is positive for $t>0$ because $\mu\ge-1$. With $t_n=u_n/v_n$ in lowest terms,
\[
u_{n+1}=F_\mu(u_n,v_n):=u_n^2+\mu u_nv_n+v_n^2,\qquad v_{n+1}=u_nv_n,
\]
again without cancellation. Exactly as in Proposition~\ref{prop:key}, $a_{n+1}=hv_n/d_n$ and $b_{n+1}=u_n/d_n$ with $d_n=\gcd(h,u_n)$, and for $n\ge1$
\begin{equation}\label{eq:mukey}
b_{n+1}+\frac{2-\mu}{h}\,a_{n+1}=\frac{(u_{n-1}+v_{n-1})^2}{d_n},\qquad
b_{n+1}-\frac{2+\mu}{h}\,a_{n+1}=\frac{(u_{n-1}-v_{n-1})^2}{d_n}.
\end{equation}
The argument of Step~1 of Section~\ref{sec:main} shows that $c=(\pm2-\mu)/h$ are the only constants for which $b_n+c\,a_n$ could always be a square.

\begin{theorem}\label{thm:mu}
Let $\mu\ge-1$, $\mu\ne2$, and $\delta=\mu-2$. Then $b_n+\frac{2-\mu}{h}a_n$ is a perfect square for every $x_0\in\N$ and every $n\ge2$ if and only if every prime $p\mid h$, with $\alpha=v_p(h)$, satisfies one of the following:
\begin{enumerate}[label=\textup{(\Alph*)},leftmargin=*]
\item $p\nmid\delta$, and $X^2+\mu X+1$ has no root modulo $p$;
\item $p\mid\delta$ with $r=v_p(\delta)$, and either $\alpha\le r$ and $\alpha$ is even, or $\alpha>r$, $r$ is even, $p$ is odd and $\delta/p^r$ is a quadratic non-residue modulo $p$.
\end{enumerate}
The same statement holds for $b_n-\frac{2+\mu}{h}a_n$ with $\delta=-(\mu+2)$. For $\mu=2$, $b_n$ is always a perfect square if and only if $h$ is a perfect square.
\end{theorem}

For odd $p\nmid\mu^2-4$, condition (A) says that $\legendre{\mu^2-4}{p}=-1$; for $p=2$ it says that $\mu$ is odd.

\begin{proof}
Fix $n\ge1$, let $(u,v)=(u_{n-1},v_{n-1})$, $d=\gcd(h,F_\mu(u,v))$ and $Q=(u+v)^2/d$. The key identity is $F_\mu(u,v)=(u+v)^2+\delta uv$. A positive rational number is a perfect square if and only if all its valuations are even and nonnegative, and $v_p(Q)=2v_p(u+v)$ for $p\nmid h$. So fix $p\mid h$.

\emph{Primes $p\nmid\delta$.} If $p\mid u+v$, then $p\nmid uv$ and $F_\mu\equiv\delta uv\not\equiv0\pmod p$. Hence $p\mid d$ forces $p\nmid u+v$ and $v_p(Q)<0$, so $Q$ is fine at $p$ if and only if $p\nmid F_\mu(u,v)$. Since $p\mid v$ gives $F_\mu\equiv u^2\not\equiv0$, pairs with $p\mid F_\mu(u,v)$ exist if and only if $X^2+\mu X+1$ has a root modulo $p$.

\emph{Primes $p\mid\delta$.} Now $F_\mu\equiv(u+v)^2\pmod p$, so $p\mid F_\mu$ if and only if $w:=v_p(u+v)\ge1$, in which case $u,v$ are units with $uv\equiv-u^2\pmod p$. Put $z=(u+v)/p^w$ and $\delta'=\delta/p^r$. Then
\[
v_p(F_\mu)=\begin{cases}2w,&2w<r,\\ r,&2w>r,\\ r+s,&2w=r,\end{cases}\qquad s=v_p(z^2+\delta'uv)\ge0,
\]
and $v_p(Q)=2w-\min(\alpha,v_p(F_\mu))$. As $z^2+\delta'uv\equiv z^2-\delta'u^2\pmod p$, one can have $s\ge1$ if and only if $\delta'$ is a square modulo $p$ when $p$ is odd, while $s\ge1$ always when $p=2$. If $\alpha<r$, choosing $2w>r$ gives $v_p(Q)=2w-\alpha$, so $\alpha$ must be even, and then $v_p(Q)$ is even and nonnegative in every case. If $\alpha=r$, the possible values are $0$, $2w-r$ and $0$, so $r$ must be even. If $\alpha>r$, the case $2w>r$ gives $2w-r$, so $r$ must be even, and the case $2w=r$ gives $v_p(Q)=-\min(\alpha-r,s)$, which is admissible if and only if $s=0$ always, i.e.\ $p$ is odd and $\delta'$ is a non-residue.

\emph{Realizability.} Every pair needed above can be taken of the form $(u,v)=(e,1)$ with $e\in\N$, and these occur at $n=1$ for $x_0=he$. The form $b-\frac{2+\mu}ha$ follows by replacing $v$ with $-v$, so that $u-v=e-1$ plays the role of $u+v$ and one takes $e=p^wz+1$. The case $\mu=2$ is $\delta=0$.
\end{proof}

\begin{corollary}\label{cor:mu}
\begin{enumerate}[label=\textup{(\alph*)},leftmargin=*]
\item $\mu=0$: Theorem~\ref{thmA}; the prime $2$ is excluded by (B) with $\delta=-2$.
\item $\mu=1$: for $x_{n+1}=h^2/(x_0+\dots+x_n+nh)$, $b_n+\frac1ha_n$ is always a perfect square if and only if every prime factor of $h$ is $\equiv2\pmod3$.
\item $\mu=-1$: for $x_{n+1}=h^2/(x_0+\dots+x_n-nh)$, $b_n+\frac3ha_n$ is always a perfect square if and only if every prime factor of $h$ is $\equiv2\pmod3$.
\item $\mu=3$: $b_n-\frac1ha_n$ is always a perfect square if and only if every prime factor of $h$ is $\equiv\pm2\pmod5$.
\item $\mu=4$: $b_n-\frac2ha_n$ is always a perfect square if and only if every prime factor of $h$ is $\equiv\pm5\pmod{12}$.
\end{enumerate}
\end{corollary}

\begin{remark}\label{rem:fields}
Every real quadratic field $K$ contains a unit $\varepsilon>1$ of norm $+1$, and then $K=\Q(\sqrt{\mu^2-4})$ with $\mu=\varepsilon+\varepsilon^{-1}\ge3$; so the recursions \eqref{eq:murec} reach every real quadratic field, while the only imaginary ones are $\Q(i)$ ($\mu=0$) and $\Q(\sqrt{-3})$ ($\mu=\pm1$). The exceptional primes in (B) do occur: for $\mu=6$ and $h=4$ the prime $2$ satisfies (B) with $\alpha=r=2$ although $X^2+6X+1\equiv(X+1)^2\pmod 2$; for $\mu=11$, $h=9$ works but $h=27$ does not; and for $\mu=7$ the form $b-\frac9ha$ works for $h=27$ because $\delta=-9$ and $-1$ is a non-residue modulo $3$.
\end{remark}

\begin{problem}[an Eisenstein companion to Problem~\ref{prob:main}]
Let $h\in\N$. For which $h$ is there a rational $c>0$ such that, for every sequence of positive rationals with $x_0\in\N$ and $x_{n+1}=h^2/(x_0+\dots+x_n+nh)$, the number $c\,a_n+b_n$ is a perfect square for every $n\ge2$? (Answer: $c=1/h$, and every prime factor of $h$ must be $\equiv2\pmod3$.)
\end{problem}

%======================================================================
\section{Computations and questions}\label{sec:final}

All statements were also tested by exhaustive computer search: Theorem~\ref{thmA} for $h\le60$ and $x_0<h^2+2h+60$; Theorem~\ref{thmC} for $k\le1000$; Theorem~\ref{thm:mu} for $-1\le\mu\le20$, both signs and $h\le90$ ($3960$ cases); and the identities \eqref{eq:key}, \eqref{eq:mukey} and Proposition~\ref{prop:orbit} for $n\le9$. No discrepancy was found. The code is available from the first author.

We end with some questions.
\begin{enumerate}[leftmargin=*]
\item For $h$ with prime factors $\equiv1\pmod4$, what proportion of starting values $x_0\le X$ lead to exactly $j$ failures, $0\le j\le\omega_1(h)$? By Remark~\ref{rem:finite} this is governed by the trees of $\F_p$-rational preimages of $0$ under $T$.
\item For the recursions of this paper the exponent $m=2$ is forced. Is there a natural family of rational recursions for which $c\,a_n+b_n$ is always a cube, other than those conjugate to $z\mapsto z^3$?
\item What happens over $\Z[i]$, over $\F_q[t]$, or for other norm forms?
\end{enumerate}

\subsection*{Acknowledgments}
A computer verification and part of the preparation of the manuscript were carried out with the assistance of an AI system (Claude, Anthropic); the authors checked all arguments and take full responsibility for the content. % Adjust to the journal's disclosure policy.

\end{document}